\documentclass{amsart}

\usepackage[english]{babel}

\usepackage{latexsym,amsfonts}
\usepackage{amssymb}
\usepackage{amsthm}

\newtheorem{lemma}{Lemma}[section]
\newtheorem{theorem}[lemma]{Theorem}
\newtheorem{cor}[lemma]{Corollary}

\theoremstyle{remark}

\def\SO{\mathop{\rm SO}\nolimits}
\def\SL{\mathop{\rm SL}\nolimits}
\def\Sp{\mathop{\rm Sp}\nolimits}
            
\def\PSL{\mathop{\rm PSL}\nolimits}          
\def\PSp{\mathop{\rm PSp}\nolimits}           
           
\def\SU{\mathop{\rm SU}\nolimits}          
\def\PSU{\mathop{\rm PSU}\nolimits}          
\def\Sp{\mathop{\rm Sp}\nolimits}

\def\Sym{\mathop{\rm Sym}\nolimits}

\newcommand{\tr}[1]{{\rm Tr}(#1)}

\def\Alt{\mathop{\rm Alt}\nolimits}

\def\id{\mathrm{id}}
\def\blockdiag{\mathop{\rm blockdiag}\nolimits}
\def\diag{\mathop{\rm diag}\nolimits}

\newcommand{\N}{\mathbb{N}}    
\newcommand{\Z}{\mathbb{Z}}    
\newcommand{\F}{\mathbb{F}}    
\title[The simple classical groups of dimension less than $6$ which are \ldots ]{The
simple classical groups of dimension less than $6$ which are $(2,3)$-generated}

\author{M.A. Pellegrini}

\address{Dipartimento di Matematica e Fisica, Universit\`a Cattolica del Sacro Cuore, Via Musei 41,
I-25121 Brescia, Italy}
\email{marcoantonio.pellegrini@unicatt.it}

\author{M.C. Tamburini Bellani}
\address{Dipartimento di Matematica e Fisica, Universit\`a Cattolica del Sacro Cuore, Via Musei 41,
I-25121 Brescia, Italy}
\email{mariaclara.tamburini@unicatt.it}


\begin{document}

\begin{abstract}
In this paper we determine the classical simple groups of dimension $n=3,5$ which
are $(2,3)$-generated (the cases  $n=2,4$ are known). If $n=3$, they are
$\PSL_3(q)$, $q\ne 4$, and $\PSU_3(q^2)$, $q^2\ne 9,25$.
If $n=5$ they are $\PSL_5(q)$, for all $q$, 
and $\PSU_5(q^2)$,  $q^2\ge 9$.  Also, the soluble group $\PSU_3(4)$ is not $(2,3)$-generated.
We give explicit $(2,3)$-generators of the linear preimages, 
in the special linear groups, of the $(2,3)$-generated simple groups.
\end{abstract}

\maketitle
\section{Introduction}
A $(2,3)$-generated group is a group that can be generated by two elements of respective orders $2$ and $3$.
Alternatively, by a result of Fricke and Klein \cite{FK},
it can be defined as an epimorphic image, of order $>3$,
of the group $\PSL_2(\Z)$. 

Many finite simple and quasi-simple groups
have been investigated with respect to this property.
Apart from the infinite families $\PSp_4(2^ m)$  and $\PSp_4(3^ m)$,
the other finite classical simple groups are $(2, 3)$-generated, up to a finite number
of exceptions,  by the probabilistic results of 
Liebeck and Shalev \cite{LS}.
These exceptions, for groups of Lie type, occur in small dimensions over small fields. 
For a recent survey see \cite{Vsur}. Our notation for classical groups
is taken in accordance with  \cite{Car}: consequently we denote by $\SU_n(q^2)$
the group  appearing  also as $\SU_n(q)$ in the literature, e.g. in \cite{BHR}.

The groups $\PSL_2(q)$ were dealt with by Macbeath in \cite{M}: the only exception is
$\PSL_2(9)$. The groups $\SL_4(q)$ and $\SU_4(q^2)$ were studied in \cite{PTV}.
In dimension $4$, apart from  the above infinite series of symplectic groups, 
the only simple groups which are not $(2,3)$-generated are  
$\PSL_4(2)$, $\PSU_4(4)$, $\PSU_4(9)$. 
The groups $\SL_n(q)$, for $n=5,6,7$ were studied in \cite{DV}, \cite{T5}, \cite{T6},\
\cite{T7}: all these groups are $(2,3)$-generated.

The aim of this paper is 
to fill the gaps for classical groups of  dimension $n\le 5$,
treating in a uniform way the special linear and the unitary groups.
This also gives a partial answer to Problem n. 18.98 of \cite{KM}.

The soluble group $\PSU_3(4)$ and the simple groups
$\PSL_3(4)$ (Garbe \cite{Ga}, Cohen \cite{Co}) and $\PSU_3(9)$ (Wagner \cite{W}) are not $(2,3)$-generated. 
The same applies to $\PSU_3(25)$ by Theorem \ref{negative results} of this paper. In
Theorems \ref{positive SL3} and \ref{positive results3} we show that:
\begin{itemize}
 \item[\rm{(i)}] $\SL_3(q)$ and $\PSL_3(q)$ are $(2,3)$-generated for all  $q\ne 4$; 
\item[\rm{(ii)}] $\SU_3(q^2)$ and $\PSU_3(q^2)$ are $(2,3)$-generated for all  $q^2\ne
4, 9,25$. 
\end{itemize}
The group $\PSU_5(4)$ is not $(2,3)$-generated (Vsemirnov \cite{V}). 
On the other hand, in Theorems \ref{positive SL5} and \ref{positive results5}
we show that: 
\begin{itemize}
\item[\rm{(i)}] $\SL_5(q)$ and $\PSL_5(q)$ are $(2,3)$-generated for all  $q$; 
\item[\rm{(ii)}] $\SU_5(q^2)$ and $\PSU_5(q^2)$ are $(2,3)$-generated for all  $q^2\ge 9$. 
\end{itemize}
The proofs of our results are based on the classification of maximal subgroups of the 
finite classical groups (\cite{BHR}) which, in the Aschbacher notation,
fall into classes ${\mathcal C}_1$, \dots , $\mathcal{C}_8$ and $\mathcal{S}$.

\section{Notation and preliminary results}

Let $\F$ be an algebraically closed field of characteristic $p\ge 0$.
When $p>0$, we set $q=p^m,\ m\ge 1$.
We denote by $\sigma$ the automorphism of $\SL_n(q^2)$ defined as:
\begin{equation}\label{sigma}
\left(\alpha_{i,j}\right)\mapsto \left(\alpha_{i,j}^q\right).
\end{equation}
For a fixed $n$, we call $x,y$ two elements of $\SL_n(\F)$  whose
respective orders are $2$ and $3$.
We denote by $z$ the product $xy$, and by  $H$ the subgroup of $\SL_n(\F)$ generated by
$x,y$, i.e.,
$H=\left\langle x,y\right\rangle$.

We need $H$ to be absolutely irreducible. 
This requirement restricts the possible canonical forms of $x,y$ and their shapes, up to conjugation.

We fix certain roots of unity in $\F^\ast$, according to the notation in  Table A.
\smallskip

\begin{center}
Table A\\\medskip
\begin{tabular}{cc|cc|cc}
Root & Order& Root& Order&Root &Order
\\ \hline\\[-10pt]
$\omega$& $\frac{3}{(3,p)}$& $i$ & $\frac{4}{(4,p^2)}$&$\eta$ & $\frac{5}{(5,p)}$
\end{tabular}
\end{center}
\smallskip

For later use, we give some lower bounds of the Euler-Phi function $\varphi: \N\to \N$.

\begin{lemma}\label{Phi 3/2} Let $n$ be a natural number. Then 
$\varphi(n)>n^{2/3}$ if, and only if,
$$n\not \in\{1,2,3,4,6,8,10,12,18,24,30,42\}.$$ 
Furthermore, $\varphi(n)=n^{2/3}$ if, and only if, $n=1,8$.
\end{lemma}

\begin{proof}
Recall that if $m,n$ are two coprime integers, then $\varphi(mn)=\varphi(m)\varphi(n)$.
Let us first consider the case $n=s^a$, with $s$ a prime. Then 
$$\varphi(s^a)= s^{a-1}(s-1)> (s^a)^{2/3}\;\Leftrightarrow\; s^{a-3}(s-1)^3>1.$$
The last inequality holds if, and only if, one the following cases occurs:

(i)\, $a=1$,   $s\neq 2,3$; \quad (ii)\, $a=2,3$,  $s\neq 2$; \quad 
(iii)\, $a\geq 4$.

In particular this implies that, if $(n,6)=1$, then $\varphi(n)> n^{2/3}$. 
So, we are left to deal with the cases when $(n,6)\neq 1$. 
Suppose first $n=2^a 3^b$ with $a,b\ge 1$. Then  
$$\varphi(2^a 3^b) > (2^a 3^b)^{2/3} \;\Leftrightarrow\;
(2^a3^b)> 27 .$$
The last inequality  holds if, and only if, one of the following occurs:

(i)\, $b=1$ and $a\geq 4$; \quad (ii)\, $b=2$  and $a\geq 2$;\quad (iii)\, $b\geq 3$  and
$a\geq 1$. 

Finally, let $n=uv$, where $(u,v)=1$ and $u=2^a3^b$, with $a+b\ge 1$. By the above, if
$u\not \in
E:=\{2,3,4,6,8,12,18,$ $24\}$, then $\varphi(n)>n^{2/3}$. On the other hand, if $u\in E$,
then $\varphi(u)\geq \frac{3}{5} u^{2/3}$. Thus, it suffices to
prove that $\varphi(v)> \frac{5}{3} v^{2/3}$.
Note that, for all primes $s\geq 11$, $\varphi(s)>\frac{5}{3} s^{2/3}$ and so
$\varphi(s^a)>s^{2(a-1)/3}\varphi(s)>\frac{5}{3}(s^{a})^{2/3}$.
For the exceptional cases $s=5,7$, we have that $\varphi(s^a)>\frac{5}{3}(s^a)^{2/3}$ if,
and only
if, $a\geq 2$. Also, observe that $\varphi(35)>\frac{5}{3}\cdot (35)^{2/3}$.  Hence, we
are left to consider only the integers $n=uv$, where $u\in E$ and $v\in \{1,5,7\}$.
Direct computations lead to the exceptions listed in the statement.
\end{proof}

The following  corollary will be used for $n=q=p^m$. 

\begin{cor}\label{pellegrini}
For all $n\geq 14$, we have $\varphi(n^2-1)> \max\{3n+21, 4n-1\}.$ 
\end{cor}

\begin{proof}
If $n\geq 64$, we may apply Lemma \ref{Phi 3/2}. In fact, for these values of $n$, it is
easy to see that 
$(n^2-1)^2> (3n+21)^3$ and $(n^2-1)^ 2> (4n-1)^ 3$. So $\varphi(n^2-1)>\max\{3n+21,
4n-1\}$. For the values 
$14\le n \leq 63$ we deduce this inequality from direct computations.
\end{proof}

In order to exclude that $H=\langle x,y\rangle$ is contained in a maximal subgroup of class $\mathcal{C}_6$
or class $\mathcal{C}_3$ we will use respectively Lemma \ref{C6} and Lemma \ref{C3}.

\begin{lemma}\label{C6}
Let  $x,y\in \SL_r(q)$, with $r$ an odd prime. Assume that $x^2=y^3=1$.
If $(xy)^6$ is not scalar,
the projective image $\overline H$ of $H$ is not contained in any subgroup 
$\overline M$  of $\PSL_r(q)$ having the following structure:
$\overline M=N.S$, where $N$ is a normal elementary abelian subgroup of order $r^2$,
and $S\le \SL_2(r)=\Sp_2(r)$ acts by conjugation on  $N$, according to 
the natural action of $\SL_2(r)$ on ${\F_r}^2$.
\end{lemma}

\begin{proof}
Assume, by contradiction, that $\overline H\le \overline M$.
Since $r$ is odd, the only involution of $\SL_2(r)$ is the central one,  namely
$-I$. Thus $Nx$ is central in $\overline M/N\simeq S$. 
It follows $N(xy)^3=Nx^3y^3=Nx$, i.e., $(xy)^3x^{-1}=n_1\in N$.
Moreover, recalling from the statement that the action of  $Nx$  on  $N\simeq \F_r^2$ is defined by 
$n\mapsto  x^{-1}nx$, we get $x^{-1}nx=n^{-1}$, for all $n\in N$.
It follows $(xy)^3=n_1x$, with $n_1\in N$, whence $(xy)^6=n_1xn_1x=n_1n_1^{-1}=1$, in contrast with the assumption that
$(xy)^6$ is not scalar.
\end{proof}

\begin{lemma}\label{C3} For $q=p^m$, $m>1$, let $a\in \F_q^\ast$ have order
$q-1$. 
For all primes $s$
$$\F_p[a^s]=\F_q $$
except, possibly, when 
$s=\sum_{j=0}^{t-1}q_0^j$, with $q=q_0^t$ for some $t\geq 2$.

In particular the statement is true for $s=3$, $q\ne 4$ and for $s=5$, $q\ne 16$.
\end{lemma}

\begin{proof}
If $(s,q-1)=1$ our claim is obvious. So we may assume that $q-1=sk$. 
Suppose  $\F_p[a^s]=\F_{q_0}$ for some $q_0$ such that $q=q_0^t$ with $t\ge 2$. Since
$k$ is the order of $a^s$, we have $q_0-1\ge k$, whence: 
$$sk=q-1= (q_0-1)\left(\sum_{j=0}^{t-1}q_0^j\right)\ge
k\left(\sum_{j=0}^{t-1}q_0^j\right).$$
This gives $s\ge \sum_{j=0}^{t-1}q_0^j$. If $s$ divides $q_0-1$ we get
the contradiction 
$$\sum_{j=0}^{t-1}q_0^j\le s \le q_0-1.$$
Thus $s$, being a prime,  must divide $\sum_{j=0}^{t-1}q_0^j$, whence $s= \sum_{j=0}^{t-1}q_0^j$.
\end{proof}

\section{Dimension 3}\label{dim3}

Let $X,y$ be elements of $\SL_3(\F)$ such that their projective images have orders 2,3 and
generate
an absolutely irreducible subgroup $H$.  Then, $X=\omega^jx$, for some $j=0,1,2$,
with $x^2=I$.  We may replace $X$ by $x$, leaving unchanged the projective image
of $H$. The involution $x$ must be conjugate to  $\blockdiag\left(-1,\begin{pmatrix}
0&1\\
1&0\end{pmatrix}\right)$. Hence $x$ has a 2-dimensional eigenspace $W$, relative to the
eigenvalue $-1$. 
Clearly  $W\cap yW$  has dimension $1$. So
we may assume  $W=\langle e_1, e_2\rangle $ with $e_2=ye_1$, and $ye_2\not\in W$. 
Setting $ye_2=e_3$, it follows $ye_3=\lambda e_1$. The condition $\det(y)=1$ gives
$\lambda=1$.
In particular $y^3=I$. Set $xe_3=ae_1+be_2+e_3$ with $a,b\in \F$. For $p=2$, we need
$(a,b)\ne(0,0)$ 
in order that $x$ is an involution. For $p$ odd and $(a,b)=(0,0)$ the group $H$ is
isomorphic to $\Alt(4)$. Hence, for our purposes, we may assume: 

\begin{equation}\label{generators3}
x=\begin{pmatrix}
-1&0&a\\
0&-1&b\\
0&0&1
\end{pmatrix},\ (a,b)\ne (0,0), \quad 
y=\begin{pmatrix}
0&0&1\\
1&0&0\\
0&1&0
\end{pmatrix}.
\end{equation}

\begin{lemma}\label{irreducibility} 
Let $x,y$ be defined as in \eqref{generators3}.
Then $H$ is absolutely irreducible  if, and only if, the following conditions hold. For all $j\in\{0,1,2\}$:
\begin{itemize}
\item[\rm{(i)}] $p\ne 2$ and  $(a,b)\ne \left(2\omega^j, 2\omega^{2j}\right)$;
\item[\rm{(ii)}] $b\ne -a\omega^j-2\omega^{2j}$.
\end{itemize}
\end{lemma}

\begin{proof} 
$H$ fixes some  $1$-dimensional space precisely when $p\neq 2$ and $(a,b)=(2\omega^j,$ $
2\omega^{2j})$, for some $j=0,1,2$. 
In this case, setting $u= (a, b, 2)^T $ 
we have $u\ne 0$ and $xu=u$, $yu\in \left\langle u\right\rangle$.
 
$H$ fixes a $2$-dimensional space if, and only if, $H^T$ fixes  a $1$-dimensional space. 
This happens  only if  $b= -a\omega^j-2\omega^{2j}$, for some $j=0,1,2$.  In this case,
taking $w=(1,\omega^{2j},\omega^{j})^T$, we have $x^Tw=-w$ and $y^Tw=\omega^{2j}
w$.
\end{proof}

\begin{lemma}\label{monomial}  
Assume  that $H$ is absolutely irreducible and that $ab\ne 1$.
Then $H$ is not conjugate to a subgroup of the monomial group. 
\end{lemma} 

\begin{proof}
Assume that $H$ acts monomially on a basis $B=\left\{v_1, v_2, v_3\right\}$. 
We may suppose $x v_1\in \left\langle v_1\right\rangle$ and,
furthermore, $xv_1=-v_1$. The latter claim is clear if $x$ fixes all the subspaces
$\left\langle v_j\right\rangle$; it follows from $x^2=I$ and the determinant condition,
if $x$ interchanges 
$\left\langle v_2\right\rangle$ with $\left\langle v_3\right\rangle$. 
Thus $v_1=(c,d,0)^T$ for some $(c,d)\ne (0,0)$.  By the irreducibility of $H$, we
may assume $v_2=yv_1= (0,c,d)^T$, $v_3=yv_2=(d,0,c)^T$. The condition $(a,b)\ne (0,0)$ 
implies that $B$ cannot be the canonical basis (up to scalars). Thus $cd\ne 0$: in
particular
$x$ cannot fix $\left\langle v_2\right\rangle$ and $\left\langle v_3\right\rangle$. 
From $xv_2=(ad,-c+bd,d)^T$ we see that $xv_2\in \left\langle v_3\right\rangle$ if, and
only if, $ab=1$.
\end{proof}

\begin{lemma}\label{powers of z} 
Assume that $H$ is absolutely irreducible and 
that $ab\ne 1$. Set $z=xy$.

\begin{itemize}
\item[\rm{(i)}] $z^j$ is not scalar for $j\le 4$ and for $j=6$.
\item[\rm{(ii)}] If $z^5$ is scalar then both $a$ and $b$ are roots of $t^6-4t^3-1$. 

\noindent More precisely:
\begin{itemize}
\item[$\bullet$]  if $p=2$, then
$(a,b)\in\left\{\left(\omega^j,1\right),\left(1,\omega^j\right),\left(\omega^j,
\omega^j\right)\right\}$, for some $j=1,2$;
\item[$\bullet$] if $p\ne 2$, then  $b=\frac{1\pm \sqrt 5}{2}\thinspace \omega^{j}$,
$a=b\omega^{j}$ for some $j=0,1,2$.
\end{itemize}

\item[\rm{(iii)}] If $z^7$ is scalar then:

\begin{itemize}
\item[$\bullet$]  either $p=2$ and $(a,b) \in \{(\omega^j, 0),(0,\omega^j)\}$, for some
$j\in \{0,1,2\}$;
\item[$\bullet$] or $a\ne 0$, $b\ne 0$ and  both $a$ and $b$ are roots of
$$R=t^{15} - 16t^{12} + 59t^9 - 67t^6 - 37t^3 + 8.$$
$R$ splits over a field containing an element of order $\frac{21}{(p,21)}$ and factorizes
over $\Z$ as
$$(t^2+t+2)(t^3-2t^2-t+1)(t^4-t^3-t^2-2t+4)(t^6+2t^5+5t^4+3t^2+t+1).$$
\end{itemize}
\end{itemize}
\end{lemma}

\begin{proof}  We consider the action of $z$ on the canonical basis $\left\{e_1, e_2,
e_3\right\}$.

\smallskip
(i)\enskip For $j\le 4$ our claim follows easily from the assumptions, since:
$$\begin{array}{cc}
ze_3= -e_1,\hfill &z^2e_3= e_2,\hfill \\
z^3e_3= ae_1+be_2+e_3,\hfill &z^4e_3= (a b - 1)e_1+(-a + b^2)e_2+be_3.\hfill 
\end{array}$$
Assume $z^6$ to be scalar and note that $z^6e_3=f_1 e_1+f_2 e_2+f_3e_3$, where
$$f_1= -2 a^2 b + a b^3 + 2 a - b^2,\quad f_2=a^2 - 3 a b^2 + b^4 + 2 b.$$
Clearly $b\ne 0$ as $b=0$ would give also $a=0$.
From $f_1= (2a-b^2)(1-ab)$,  we get $2a=b^2$ as we are excluding $ab=1$. So $p\ne 2$ and
$a=\frac{b^2}{2}$. Substitution in
$f_2$ gives
$b^4=8b$. It follows that $(a,b)= \left( 2 \omega^{j},
2\omega^{2j}\right)$,
against the assumption that $H$ is absolutely irreducible.

\smallskip
(ii)\enskip Assume $z^5$ to be scalar and note that $z^5e_3= f_1 e_1+f_2 e_2+f_3e_3$
where
$$f_1=-a^2 + a b^2 - b, \quad f_2= -2 a b + b^3 + 1.$$
Notice that $b\ne 0$, since $b=0$ would give $a=0$.   Similarly $a\ne 0$.
Considering $f_1$ and $f_2$ as polynomials in $a$, their resultant $\mathrm{Res}(f_1,f_2)$
is
$b^2(b^6-4b^3-1)$.  By \cite[Theorem 5.7, p. 325]{Jac}, where 
the definition and the properties of resultants are given, the non-zero values of $b$
satisfying 
$f_1=f_2=0$ must be roots of
$$R=t^6-4t^3-1=(t^2-t-1) (t^2-\omega t-\omega^2)(t^2-\omega^2t-\omega).$$
Calling $t_1,t_2$ the roots of $t^2-t-1$, the roots of $P$ are  $t_k \omega^j$, 
$k=1,2$,  $j=0,1,2$. Since $b$ is a root of $P$,
we may set $b=t_k\omega^j$, whence $b^3+1= 2(t_k+1)$. 
Now we equate to zero the coefficient of $e_2$. If $p=2$, we may assume $t_1=\omega$, 
whence $b$ must be a power of $\omega$.
If $p\ne 2$, noting that $\frac{1}{t_k}= t_k-1$, we get 
$$a=\frac{2(t_k+1)}{2t_k\omega^j}= (t_k^2-1)\omega^{2j}=t_k\omega^{2j}=b\omega^j.$$

Direct computation shows that $z^5$ is scalar for the values given in the statement.

\smallskip
(iii)\enskip Assume that $z^7$ is scalar and note that $z^7e_3=f_1 e_1+f_2 e_2+f_3e_3$,
where
$$f_1=a^3 - 3 a^2 b^2 + a b^4 + 4 a b - b^3 - 1,\quad f_2= 3 a^2 b - 4 a b^3 - 2 a + b^5
+3
b^2.$$
If $a=0$, then $f_1=-b^3-1$ and $f_2=b^5+3b^2$. It follows that $p=2$ and $b^3=1$.
If $b=0$, then $f_1=a^3-1$ and $f_2=2a=0$. It follows that $p=2$ and $a^3=1$.
So, assume  $a,b\neq 0$. 
The resultant $R$ of $f_1, f_2$ with respect to $a$ (and to $b$) is:
\begin{equation}\label{resultant}
R=t^{15} - 16t^{12} + 59t^9 - 67t^6 - 37t^3 + 8.
\end{equation}
Now, we can factorize $R$ as
$$(t^2+t+2)(t^3-2t^2-t+1)(t^4-t^3-t^2-2t+4)(t^6+2t^5+5t^4+3t^2+t+1).$$
One can also verify that $R$ splits over a field containing an element of order
$\frac{21}{(p,21)}$.
\end{proof}

\begin{lemma}\label{forms} 
Assume that $H$ is absolutely irreducible.
\begin{itemize}
\item[\rm{(i)}]  If $H\le \SL_3(q^2)$, then $H\le \SU_3(q^2)$ if, and only if, $b=a^q$.
\item[\rm{(ii)}] If $q$ is odd and $H\le \SL_3(q)$, then $H\le \SO_3(q)$ if, and only if,
$b=a$.
\item[\rm{(iii)}]  Let $a=0$. There exist
no non-degenerate  hermitian form and no non-degenerate orthogonal  form fixed by $H$ up
to scalars.  
\end{itemize}
\end{lemma}

\begin{proof}
The characteristic polynomials of $z=xy$ and $z^{-1}$ are respectively:
\begin{equation}\label{carpolz}
\chi_z(t)= t^3-bt^2+at-1,\quad \chi_{z^{-1}}(t)= t^3-at^2+bt-1.
\end{equation}
The centralizers of $x$, $y$ and $z$ have respective dimensions $5$, $3$ and $\geq 3$.
As $5+3+3=9+2$ and  $H$ is absolutely irreducible, by Scott's formula (see \cite[Formula
$(2)$, page 324]{PTV}),
the centralizer of $z$ must have the smallest possible dimension, i.e., $3$.
Thus, by  \cite[Theorem 3.16, page 207]{Jac} $z$ has a unique non-constant invariant
factor
(similarity invariant), namely its characteristic
polynomial. Equivalently $z$ is conjugate to a companion matrix.
The same applies to $z^{-1}$ and to $z^{\sigma}$.

\smallskip
(i)\enskip Let $\sigma$ be as in \eqref{sigma}. If  $H\le \SU_3(q^2)$, then $z^\sigma$ is
conjugate to $z^{-1}$, whence
$b=a^q$ from \eqref{carpolz}. Conversely, if $b=a^q$, from \eqref{carpolz} and the
considerations
below it,  we get that $z^\sigma$ is conjugate  to $z^{-1}$.
By \cite[Corollary 3.2, page 326]{PTV} the group $H$
fixes a non-degenerate Hermitian form, whence our claim. 

\smallskip
(ii)\enskip If  $q$ is odd and $H\le \SO_3(q)$, then $z$ is conjugate to $z^{-1}$, whence
$b=a$.
Conversely, if $b=a$, from \eqref{carpolz} it follows that 
$z$ is conjugate  to $z^{-1}$.
By the same Corollary quoted above, $H$ fixes a non-degenerate
bilinear form, whence our claim. 

\smallskip
(iii)\enskip We argue by contradiction.  Let $J$ be a non-degenerate form
fixed by $H$ up to scalars, 
with $J^T=J^\tau$,  $\tau\in \left\{{\rm id}, \sigma\right\}$. 
From  $x^TJx^\tau=\pm J$ and $\det(-I_3)=-1$ we get (I) $x^TJx^\tau=J$. It follows that 
$J$ is non-scalar as $b\ne 0$.   Moreover
(II) $y^T J y=\omega^{\pm 1}J$, with $\omega\ne 1$ by (i), (ii) and the assumption $a=0$,
$b\ne 0$. In particular, $\omega\neq 1$ gives $p\ne 3$. Condition (II) forces
$J=\diag(1,\omega^j, \omega^{2j})\cdot (c_1y^0+c_3y+c_2y^2)$.
Condition $J^T=J^\tau$ gives $c_3=\omega^{2j}c_2^\tau = \omega^{2j\tau}c_2^\tau$.
Now
$c_2\ne 0$, as $J$ is non-scalar, hence $\omega=\omega^\tau$. 
Thus
$$J=\begin{pmatrix}
c_1&c_2&\omega^{\pm 2}c_2^\tau\\
\noalign{\smallskip}
c_2^\tau&\omega^{\pm 1}c_1&\omega^{\pm 1}c_2\\
\noalign{\smallskip}
\omega^{\pm 2}c_2&\omega^{\pm 1}c_2^\tau&\omega^{\pm 2}c_1
\end{pmatrix},\ c_1=c_1^\tau\in \F ,\ c_2\in \F^\ast.$$
Now we impose that  $D=x^TJx^\tau-J$ is zero. Assume first $c_1=0$.
From $D_{2,3}=-2\omega^{\pm 1}c_2$ it follows $p=2$. But, in this case,
$D_{1,3}=c_2b^\tau$ gives the 
contradiction $b=0$.
So we may suppose $c_1=1$. 
From $D_{2,3}=0$ we get $b=-2c_2^\tau$, whence $p\ne 2$ and $c_2=\frac{-b^\tau}{2}$.
After this substitution, $D_{1,3}=0$ gives:
$$-2\omega^{\pm 2}=\frac{b^{2\tau}}{b}=\frac{b^{2}}{b^\tau}.$$
Thus  $b^{6\tau}=b^6=-8b^3$ gives $b^3=-8$. 
By Lemma \ref{irreducibility}(ii) $H$ is reducible, a contradiction.
\end{proof}

\begin{theorem}\label{positive SL3}  
Let $x,y$ be defined as in \eqref{generators3} with $a=0$, $b\in \F_{q}^ \ast$ such
that:
\begin{itemize}
\item[\rm{(i)}]  $b\ne -2\omega^j$, $\forall\ j\in \{0,1,2\}$;
\item[\rm{(ii)}] $\F_{q}=\F_p[b^3]$;
\item[\rm{(iii)}] $b\neq \omega^j$, $j\in \{0,1,2\}$ when $p=2$.
\end{itemize}
Then $H=\SL_3(q)$.
Moreover for $q\geq 5$, there exists $b\in \F_q^\ast$ satisfying {\rm (i)} to {\rm (iii)}
above.
For $q=2,3$, taking $(a,b)=(1,2)$ in \eqref{generators3}, one gets $H=\SL_3(q)$.

In particular, the groups
$\SL_3(q)$ and $\PSL_3(q)$
are $(2,3)$-generated for all  $q\ne 4$. 
\end{theorem}

\begin{proof} 
Let $M$ be a maximal subgroup of $\SL_3(q)$ which contains $H$. 
By Lemma \ref{irreducibility} and assumption (i) the group $H$ is absolutely
irreducible. In particular $M$ cannot
belong to $\mathcal{C}_1\cup\mathcal{C}_3$. Since $ab= 0$, by  Lemma
\ref{monomial}, the group $M$  cannot belong to $\mathcal{C}_2$. 
By point (i) of Lemma \ref{powers of z} and Lemma \ref{C6} the group $M$
cannot belong to $\mathcal{C}_6$.  
We note that $z=xy$ has trace $b$. So, by assumption (ii), $H$ is not conjugate to any
subgroup of $\SL_3(q_0)Z$, with $q_0<q$ and $Z$ the center of $\SL_3(q)$.
Thus $M$ cannot belong to $\mathcal{C}_5$.   
If $M$ is in class $\mathcal{S}$, then its
projective image  should be either $\Alt(6)$
or $\PSL_2(7)$.
But this possibility cannot arise as,
by  (iii) and Lemma \ref{powers of z}, the projective order of $z$ is $\ge 8$.
By point (iii) of Lemma \ref{forms} the group $M$ does not belong to the class
$\mathcal{C}_8$ and 
so we conclude that
$H=\SL_3(q)$.

Now we come to the existence of $b$. 
If $q=p$, the elements $b\in \F_{q}^\ast$ which do not satisfy either (i)  or
(iii) are at most $3$. Hence, for  $q=p\ge 5$, there exists at least one $b\in \F_p^ \ast$
satisfying  conditions (i) and (ii).
If $q \neq p$, it suffices to find an element $b\in \F_q^ \ast$ satisfying (ii). 
This is possible for $q\ne 4$: indeed we may take $b$ of order $q-1$ since, by Lemma
\ref{C3},
$\F_p[b^3]=\F_p[b]$.

Finally let $q=2,3$, $(a,b)=(1,2)$. The classes $\mathcal{C}_5$ and $\mathcal{S}$ are
empty. 
Also the class
$\mathcal{C}_8$ is empty for $q=2$. For $q=3$, assume that $M$ is a subgroup of
$\SO_3(3)\cong \Sym(4)$. By Lemma \ref{powers of z} this cannot happen, since the order of
$xy$ is greater than $4$. We conclude that $H=\SL_3(q)$.
\end{proof}

\begin{theorem}\label{negative results} 
For $p=2,3,5$ the groups $\PSU_3(p^2)$
are not $(2,3)$-generated.
\end{theorem}

\begin{proof}
 By what observed at the beginning of this Section,
it is enough to show that, for $p=2,3,5$,
 it is not possible to find $a,b$ in the
definition of $x$ so that $H=\SU_3(p^2)$. 
By contradiction, suppose that $a,b$ exist.

Point (i) of Lemma \ref{forms} gives $b=a^p$. Clearly $a^p\ne a$.

\textbf{Case  $q^2=25$.} Set $\F_{25}=\left\{h+k\omega \mid h,k\in \F_5\right\}$.
For $a\in \left\{\omega, -1-\omega\right\}$, $ab=1$ and $H$ is monomial by Lemma
\ref{monomial}.
For $a\in\left\{2\omega, -2-2\omega\right\}$, $H$ is reducible by point
(i) of Lemma \ref{irreducibility}.
For $a\in \left\{3\omega, -3-3\omega\right\}$, we have $z^5$ scalar, thus $H$ has 
$\Alt(5)$ as composition factor.
For $a\in \left\{\pm (3+\omega), \pm(3+2\omega), \pm(1+3\omega)\right\}$ we have $z^7$
scalar and $[x,y]^4=I$,
thus $H$ has  $\PSL_2(7)$ as composition factor.
For the remaining values, namely for 
$a\in \left\{-\omega, 1+\omega,
1+2\omega,1-\omega,2+\omega,3-\omega,-1+\omega,-1+3\omega\right\}$, the group
$H$ is reducible by point (ii) of
Lemma \ref{irreducibility}.

\textbf{Case $q^2=9$.} Set $\F_9=\F_3(\alpha)$, $\alpha^2=\alpha +1$.
If $a\in \{\alpha , 2\alpha+1\}$,  then $H$ is reducible.
If $a\in \{\alpha+1 , \alpha^3+1\}$, then  $ab=a^4=1$ and $H$ is monomial. 
If $a\in \{-\alpha , -\alpha^3\}$, then $z=xy$ has order $7$ and $[x,y]$ has order 4.
It follows that $H\simeq \PSL_2(7)$.

\textbf{Case $q^2=4$.} Then $a=\omega^{\pm 1}$  and
$H$ is reducible by point (ii) of Lemma \ref{irreducibility}.
\end{proof}

\begin{theorem}\label{positive results3}  
Suppose $q^2\ne 4,9,25$. Define $x,y$ as in \eqref{generators3}
with $a\in \F_{q^2}^{\ast}$, $b=a^q$, and the further conditions:
\begin{itemize}
\item[\rm{(i)}] $a^q+a\omega^j+2\omega^{2j}\ne 0$, $j\in \{0,1,2\}$;
\item[\rm{(ii)}]  $a^{q+1}-1\ne 0$;
\item[\rm{(iii)}] $\F_{q^2}=\F_p[a^3]$;
\item[\rm{(iv)}] if $p\ne 2$, then  $a^6-4a^3-1\ne 0$;
\item[\rm{(v)}] $a^{15} - 16a^{12} + 59a^9 - 67a^6 - 37a^3 + 8\ne 0$.
\end{itemize}

Then $H=\SU_3(q^2)$. Moreover there exists $a\in \F_{q^2}$ satisfying {\rm (i)} to {\rm
(v)} above.
In particular, $\SU_3(q^2)$ and $\PSU_3(q^2)$
are $(2,3)$-generated for all  $q^2\ne 4,9,25$. 
\end{theorem}

\begin{proof}
By point (i) of Lemma  \ref{forms} and the assumption $b=a^q$ we have that
$H\le \SU_3(q^2)$. 
Let $M$ be a maximal subgroup of $\SU_3(q^2)$ which contains $H$. 
By Lemma \ref{irreducibility} and conditions (i) and (iii), $H$ is absolutely
irreducible.  In particular $M$ cannot
belong to $\mathcal{C}_1\cup\mathcal{C}_3$. By (ii) we have $ab\ne 1$ so
that $M$ cannot belong to $\mathcal{C}_2$  by Lemma \ref{monomial}.
Since $z^6$ is not scalar, by Lemma \ref{powers of z}, we may apply Lemma
 \ref{C6} with $r=3$ to deduce that $M$ does belong to $\mathcal{C}_6$.
From $b=a^q$ it follows $\F_p[b^3]=\F_p[a^3]=\F_{q^2}$. As $z$ has trace $b$, it cannot be
conjugate 
to any element of $\SL_3(q_0)Z$, with $q_0< q^2$ and $Z$ the center of $\SU_3(q^ 2)$.
Thus $M\not\in \mathcal{C}_5$. 
Finally,  if $M$ is in class $\mathcal{S}$,
then its
projective image  should be either $\Alt(6)$
or $\PSL_2(7)$, since we are excluding $q=5$.
But this possibility cannot arise as,
by  (iv), (v) and Lemma \ref{powers of z} , the projective order of $z$ is $\ge 8$.
We conclude that $H=\SU_3(q^2)$.

As to the existence of some $a$ satisfying all the assumptions, we first note that 
the elements of $\F_{q^2}$ which do not satisfy either (i) or (iv) or (v) are
at most  $3q+6+15=3q+21$. On the other hand, any $a\in \F_{q^2}^ \ast$ of order $q^2-1$,  
satisfies (ii) and also (iii) by Lemma \ref{C3}. 

Assume first $q\ge 16$. By Lemma \ref{pellegrini} there are at least $3q+22$ elements
$a\in\F_{q^2}^\ast$ having order $q^2-1$, and our claim is true.

For $q<16$ we may take $a$
of order $q^2-1$ and minimal polynomial specified in Table B. Indeed this minimal
polynomial is coprime with
each of $t^q+t\omega^j+2\omega^{2j}$, for all $j$,  
$t^6-4t^3-1$ and $t^{15} - 16t^{12} + 59t^9 - 67t^6 - 37t^3 + 8$.

\begin{center}
Table B\\\medskip
\begin{tabular}{cc|cc}
$q$ & Minimal Polynomial&$q$&Minimal Polynomial \\ \hline\\[-10pt]
$4 $& $t^4+t+1$&$9$& $t^4+t^3+t^2+2t+2$\\
$7$& $t^2+6t+3$&$11$& $t^2+7t+2$\\
$8$& $t^6 + t^4 + t^3 + t + 1$&$13$&$t^2-t+2$.
\end{tabular}
\end{center}
\end{proof}

\section{Dimension 5}\label{dim5}

In \cite[Lemma 2.1, page 124]{V} M. Vsemirnov parametrizes, up to conjugation,
all the $(2,3)$ irreducible pairs $X,Y\in \SL_5(\F)$, for $p\ne 3$.
His purpose is to show that $\PSU_5(4)$ is not $(2,3)$-generated.
In his parametrization, taking  $a=-1$, $d=0$, 
$b,c\in \F$, we obtain 

\begin{equation}\label{generators5}
x=\begin{pmatrix}
0&1&0&0&c\\
1&0&0&0&-c\\
0&0&0&1&0\\
0&0&1&0&0\\
0&0&0&0&1
\end{pmatrix},\, c\neq 0,\quad 
y=\begin{pmatrix}
1&0&-1&0&b\\
0&0&-1&0&0\\
0&1&-1&0&0\\
0&0&0&0&-1\\
0&0&0&1&-1
\end{pmatrix}.
\end{equation}

\begin{lemma}\label{irred5}
Let $x,y$ be as in $\eqref{generators5}$.
Then $H=\langle x,y\rangle$ is absolutely irreducible if, and only if, the following
conditions hold:
\begin{itemize}
\item[\rm{(i)}] $b^2 + 3bc - b + 3c^2 - 3c + 1 \ne 0$, i.e., $b\ne (\omega^{\pm
1}-1)c-\omega^{\pm 1}$;
\item[\rm{(ii)}] $b^2+10b+16c+9\ne 0$, i.e., $(b,c)\ne (4\gamma-1, -\gamma^2 -2\gamma)$,
$\forall\ \gamma\in \F$.
\end{itemize}
\end{lemma}

\begin{proof}
By Lemma 2.2 of \cite{V}, if $b^2 + 3bc - b + 3c^2 - 3c + 1 = 0$ the group $H$ is
reducible
over $\F$. So condition (i) is necessary. 
If $(b,c)=(4\gamma-1,-\gamma^2 -2\gamma)$ for some $\gamma\in \F$,  there is a
$2$-dimensional  $H$-invariant space, namely $\left\langle w, yw\right\rangle $,
where $w=(\gamma , -\gamma, 1 , -1, 0)^T$. 
So condition (ii) is also necessary.

Now let $W$ be a proper $H$-invariant subspace of $V=\F^ 5$. Denote by 
$V_1^x$, the eigenspace of $x$ relative to $1$. Notice that
$$ xe_1=e_2,\quad yxe_1=e_3,\quad xyxe_1=e_4,\quad (yx)^ 2 e_1=e_5.$$
Thus $e_1\not\in W$, since $W\neq V$. A direct calculation shows that $(I+y+y^2)V=\langle
e_1\rangle$.
It follows that every element $w\in W$ satisfies the condition 
\begin{equation}\label{1+y+y^2}
w+yw+y^ 2w=0. 
\end{equation}
{\bf Case 1} $W$ is not contained in $V_1^x$. 
So there exists a non-zero vector $u\in W$ such that $w=u-xu\ne 0$.
Calculation of $u-xu$ for a generic $u\in \F^5$ gives $w=\left(x_1,
-x_1,x_3,-x_3,0\right)^T$,
for some $x_1,x_3\in \F$.
The vector $w$ satisfies condition \eqref{1+y+y^2} if, and only if, $4x_1-(b+1)x_3=0$.
If $p\neq 2$, then we set $x_1=\frac{b+1}{4}x_3$.
Now, also $xyw$ belongs to $W$ and so it  must satisfy condition \eqref{1+y+y^2}.
By condition (ii), we get the contradiction $w=0$.
If $p=2$, we get $(b+1)x_3=0$. Condition (ii) implies $x_3=0$. In this case condition
\eqref{1+y+y^2}, applied to 
$xyw$, gives again $w=0$.

\smallskip
{\bf Case 2} $W$ is contained in $V_1^x$. Let $0\ne w=(cx_1,cx_2,x_3,x_3,x_1-x_2)\in W$.
From $xyw=yw$ we get $x_3=x_1+(c-1)x_2$ whence $x_1(b+c)-x_2(b+c^ 2)=0$,
eliminating $x_3$.
If $b+c\neq 0$, we set $x_1=\frac{b+c^ 2}{b+c}x_2$. Using condition (i) and the assumption
 $c\ne 0$,  we see
that  $w$ satisfies   \eqref{1+y+y^2} only if it is zero, a contradiction.
Finally assume $b= -c$. Then $xyw=yw$ gives $x_1=x_3$ and $(c-1)x_2=0$. Moreover
$xy^2w=y^2w$
gives $(c-1)x_1=0$. From $w\ne 0$ we conclude  $c=1$, $b=-1$, excluded by (i).
\end{proof}

\begin{lemma}\label{monomial5} 
If $H$ is absolutely irreducible, then it is not monomial.
\end{lemma}

\begin{proof}
Assume that $H$ acts monomially on a basis $B=\left\{v, \dots\right\}$. 
By the irreducibility of $H$, 
the permutations induced by $x,y$ must generate a transitive
subgroup of $\Sym(5)$. In particular  $x$ must act as the product of two 2-cycles, 
$y$ as a 3-cycle. This gives $p\ne 3$ as, for $p=3$, the eigenspace of $y$ relative to $1$
has dimension $2$. We may suppose that $xv=\pm v$.
From  $x^2=1$ and $\det(x)=1$, we get $xv=v$, hence
$ v=(cx_1,cx_2,x_3,x_3,x_1-x_2)^T$.
It follows that:
$$B=\left\{v,yv,y^2v,xyv,xy^2v\right\}.$$
So $B$ must be independent and, for some $j\in\left\{1,2\right\}$,
both vectors
\begin{equation}\label{fixy}
y(xyv)-\omega^{j} (xyv),\quad y(xy^2v)-\omega^{-j} (xy^2v),\quad \omega\ne 1 
\end{equation}
have to be zero, since $y$ must have the same similarity invariants of
\eqref{generators5}.
Coordinates $4$ and $5$ of the vectors in \eqref{fixy} are $0$ only if 
$$x_3= -\omega^{-j} cx_2,\quad x_1= (2\omega^jc +2c+1)x_2.$$
After these substitutions, putting $x_2=1$, coordinates 2 and 3 of the two vectors are $0$
only if
$(b,c)=\left(-1, \frac{-\omega^j+ 1}{3}\right)$.  For $p=2$, $H$ is reducible 
by (ii) of Lemma \ref{irred5}.  For $p$ odd,  coordinate $1$ of 
the second vector in \eqref{fixy}  is $-2$, a contradiction. 
\end{proof}

\begin{lemma}\label{forms2}
Suppose $H$ absolutely irreducible.
\begin{itemize}
\item[\rm{(i)}] if $q$ is odd and $H\le\SL_5(q)$, then $H\le \SO_5(q)$
if, and only if, $b=-1$;
\item[\rm{(ii)}] if $H\le\SL_5(q^2)$, then $H\le \SU_5(q^2)$
if, and only if, $b=c^q-c-1$.
\end{itemize}
\end{lemma}

\begin{proof}
The characteristic polynomials of $z=xy$ and $z^{-1}$ are respectively 
\begin{equation}\label{char5}
t^5 + t^4 + ct^3 + (-b - c - 1)t^2 -t-1, \quad 
t^5 + t^4 + (b + c + 1)t^3 -ct^2 -t -1. 
\end{equation}
The centralizers of $x$, $y$ and $z$ have respective dimensions $13$, $9$ and $\geq 5$.
From  $13+9+5=25+2$ and $H$ absolutely irreducible, by Scott's formula we get that 
the centralizer of $z$ must have dimension $5$ (e.g. see \cite{PTV}).
Thus  $z$ has a unique non-constant invariant factor,
namely its characteristic polynomial. Equivalently $z$ is conjugate to a companion
matrix.
The same applies to $z^{-1}$ and to $z^{\tau}$,
where $\tau\in\left\{\id, \sigma\right\}$. 
Thus, if $b=c^\tau-c-1$, comparing the characteristic polynomials we get that
$z^\tau$ is conjugate  to $z^{-1}$.
By \cite[Corollary 3.2, page 326]{PTV} the group $H$
fixes a non-degenerate orthogonal or Hermitian form, whence our claim. 
\end{proof}

\begin{lemma}\label{5powers}
Assume $H$ absolutely irreducible. Then
\begin{itemize}
 \item[\rm{(i)}] if $(xy)^k$ is scalar for some positive integer $k$, then $k\geq 10$;
 \item[\rm{(ii)}] if $[x,y]^k$ is scalar for some positive integer $k$, then $k\geq 5$.
\end{itemize}
\end{lemma}

\begin{proof}
(i)\enskip Let $D=(xy)^k$. For $k=1,\ldots,5$, it suffices to look at the third row of $D$ to conclude that $D$ is not scalar. For $k=6$, the condition $D_{3,1}=D_{3,2}=0$ implies $b=-3$ and $c=1$. So, it is easy to see that $D$ is scalar only if $p=2$, but in this case $H$ is reducible.
For $k=7$, the condition $D_{3,1}=D_{3,2}=0$ implies $(b,c)=(3,-1)$. However, in
this cases $D$ is not scalar.
For $k=8$, the condition $D_{3,1}-D_{5,1}=0$ implies $(2c-1)(b+2c+1)=0$. If $p=2$, then
$b=1$ and $D_{4,1}=0$ implies $c=0$. 
Assume $p\neq 2$. If $c=\frac{1}{2}$, then
$D_{3,1}=0$ implies $b=-\frac{9}{8}$. It is easy to see now that $D$ is  not scalar for
these values of $b,c$. If $b=-2c-1$, then the condition $D_{3,4}-D_{3,5}=0$ implies $c=2$.
Also in this case one can verify that $D$ is not scalar.
For $k=9$, the condition $D_{3,2}-D_{5,2}=0$ implies $(b+2c+1)(2b-3c^ 2+6c-1)=0$. If
$b=-2c-1$, then $D_{4,5}=0$ implies $c^2=-2$. In this case, it is easy to see that $D$
is not scalar. Now, assume 
$2b-3c^ 2+6c-1=0$. If $p=2$, this means $c=1$. In this case $D_{3,1}=0$ implies $b=1$, but $D$ is not scalar. If $p\neq 2$, take $b=\frac{3c^ 2-6c+1}{2}$. We have $D_{3,1}=3c^3-\frac{11}{2}c^2+5c-\frac{5}{2}=0$. If $p=3$, we obtain $c^2+c+1=0$, i.e. $c=1$. In this case $H$ is reducible. So, assume $p\neq 3$. In this case $D_{3,1}=(3c-3)(c^ 2-\frac{5}{6}c+\frac{5}{6})$. It is now easy to see that  $D$ is not scalar.

\smallskip
(ii)\enskip Let  $D=[x,y]^k$. For $k=1,2$, it suffices to look at the fifth row of $D$ to
conclude that $D$ is not scalar. For $k=3$, the condition $D_{3,1}-D_{5,1}=0$ implies
$c=0$.
For $k=4$, the condition $D_{3,1}-D_{5,1}=0$ implies $c^2+2b+3c=0$. If $p=2$,
 we obtain $c=1$. 
In this case, $D$ is scalar only if
$b=1$. 
However, if $b=c=1$, then $H$ is reducible. Assume $p\neq 2$. Take  $b=\frac{-c^2-3c}{2}$. 
The condition $D_{3,2}-D_{5,2}=0$ implies $c=1,-1$. In these cases $D$ is not scalar.
\end{proof}

\begin{lemma}\label{class S}
Assume $H$ absolutely irreducible. Then $H$ is not contained in any 
subgroup $M=\langle \eta I\rangle \times S$, where $S$ is isomorphic to  $\PSU_4(2)$ or $\PSL_2(11)$.  
\end{lemma}

\begin{proof}
First, we observe that since $H$ is absolutely irreducible, both $x$ and $y$  actually belong to the subgroup $S$, so $H\leq S$. Furthermore, by Lemma \ref{5powers}, we may assume that $xy$ and $[x,y]$ have  order at least $10$ and $5$, respectively. Also, we recall that $\tr x=1$ and $\tr y= \tr {xy}=-1$. 

Assume that $S$ is isomorphic to $\PSU_4(2)$. Then, $z=xy$ must be an element of order
$12$. This implies that $x \in 2B$ and $y\in \{3A,3B \}$ in Atlas notation (simply, use
the character table).
On the other hand one can verify that in this case the order of $[x,  y]$ is less than
$5$. This contradicts Lemma \ref{5powers}.

Assume that $S$ is isomorphic to $\PSL_2(11)$. Then, $xy$ must be an element of order $11$. 
We consider the character of an irreducible representation of degree $5$ of $\PSL_2(11)$ over $\F$. 
If $p\neq 2,3,11$, then $\chi(xy)=\frac{-1\pm \sqrt{-11}}{2}$ which leads to $-2=-1\pm \sqrt{-11}$, i.e. $12=0$, in contradiction with the hypothesis on $p$. 
For $p=2$, we have $\chi(xy)=\omega^{\pm 1}$, a contradiction.  For $p=3$ we have $\tr{[x,y]}=0 = b+2c$. In this case, $z_{4,1}=c(c+1)(c-1)$. 
However, for $c=\pm 1$ the matrix $(xy)^{11}$ is not scalar.
Finally, for $p=11$ we obtain $\tr{xy}=-1=\chi(xy)=5$, a contradiction.
\end{proof}

\begin{theorem}\label{positive SL5}  Let $x,y$ as in \eqref{generators5} with $b=0$ and
$c\in \F_q^\ast$ such that:
\begin{itemize}
\item[\rm{(i)}]  $3c^2-3c+1\ne 0$ and  $16c+9\ne 0$;
\item [\rm{(ii)}] $\F_p[c^5]=\F_q$;
\item [\rm{(iii)}] if $p=2$ and $q$ is a square, then $c^{\sqrt q}+c+1\ne 0$.
\end{itemize}
Then $H=\SL_5(q)$. Moreover, if $q\neq 4,16$, there exists $c\in \F_q^\ast$ satisfying
{\rm (i), (ii), (iii)}. 
 If $q=4$, taking $b=c=\omega$ in \eqref{generators5},
on gets $H=\SL_5(4)$ and if $q=16$, there exists $c\in \F_q^ \ast$ such that $H=\SL_5(16)$. In particular, the groups
$\SL_5(q)$ and $\PSL_5(q)$
are $(2,3)$-generated for all  $q$. 
\end{theorem}

\begin{proof}
Let $M$ be a maximal subgroup of $\SL_5(q)$ which contains $H$. 
By Lemma \ref{irred5} and condition (i), $H$ is absolutely
irreducible.  In particular $M\not\in \mathcal{C}_1\cup\mathcal{C}_3$. Moreover
$M\not\in \mathcal{C}_2$  by Lemma \ref{monomial5}.
Since $z^6$ is not scalar by Lemma \ref{5powers}, we may apply Lemma
\ref{C6} with $s=5$ to deduce that $M\notin\mathcal{C}_6$.
Since $c$ is a coefficient of the characteristic polynomial of $z$ (see \eqref{char5})
and
$\F_p[c^5]=\F_{q}$ by (ii), the matrix $z$ cannot be conjugate 
to any matrix $(\lambda I) z_0$, with  $(\lambda I) \in \SL_5(\F)$ and $z_0\in \SL_5(q_0)$,
for any $q_0< q$. Indeed, in this case, a coefficient of the characteristic polynomial of $z_0$ is 
$c\lambda^{-1}\in \F_{q_0}$, whence $c^ 5=(c\lambda^{-1})^5\in \F_{q_0}$, in contrast with (ii).
We conclude that  $M$ does not belong to class $\mathcal{C}_5$. 

We now exclude that $M\in \mathcal{C}_8$. Let $J$ be a non singular matrix such that
$x^T J x ^ \tau = \lambda J$ and  $y^T J y ^ \tau = \mu J$, 
for some $\lambda, \mu \neq 0$. Considering the orders of $x,y$
we get $\lambda^2=\mu^3=1$. On the other hand, considering the determinants,
we obtain $\lambda^5=\mu^5=1$. We conclude that $\lambda=\mu=1$. In other words, if $H$
fixes $J$ up to scalars,
then it actually fixes $J$. 
By Lemma \ref{forms2}(i) $H$ does not fix any orthogonal form. On the other hand, assume
that $q$ is a square
and that $H\leq \SU_5(q)$. For $b=0$, Lemma \ref{forms2}(ii) gives $c^{\sqrt q}=c+1$.
It follows $c=c^q=(c+1)^ {\sqrt q}=c+2$, whence $p=2$. So this case is excluded by assumption (iii).

Finally, suppose $M$ in class $\mathcal{S}$. By Lemma \ref{class S}
we have to consider only the case $q=3$ and $M=M_{11}$. However, if $c=\pm 1$, 
then $xy$ has order, respectively, $16$ and $121$, a contradiction. We conclude that $H=\SL_5(q)$.

As to the existence of some $c$ satisfying all the assumptions in the statement,  
the elements of $\F_{q}^\ast$ which do not satisfy  (i)  are
at most  $3$. Moreover,  by Lemma \ref{C3}, if $q\ne 16$, any element in $\F_q^ \ast$ of order $q-1$
satisfies (ii). Hence, if $p$ is odd and $q\geq 9$, there exists $c$ of order $q-1$ satisfying
all conditions since $\varphi(q-1)>3$. 
If $q=3,5,7$, one can check that $c=-2$ satisfies  (i) and (ii).

Now assume  $p=2$. If $q$ is not a square, any element in  $ \F_q^ \ast$ of order $q-1$ satisfies 
all the assumptions.
If $q$ is a square, assume first $q\ne 2^2, 4^2$. 
The elements in  $\F_q^\ast$ of order $q-1$, which do not satisfy either (i) or (iii), are at most 
$\sqrt{q}+2$. Moreover all elements of order $q-1$ satisfy (ii). We have
$\varphi(q-1)> 4 \sqrt{q} > \sqrt{q}+2 $ by Corollary \ref{pellegrini}
and direct calculation if $q<16^2$. Thus $c$ exists if $q\ne 2^2,4^2$ and so $H=\SL_5(q)$.
When $q=4^2$,  we take $c$ of order $15$, with minimal polynomial $t^4+t^3+1$. Then $c$ satisfies (i) and (iii).
Moreover $z=xy$ has order divisible by $41$, a prime which does not divide $\left|\SL_5(2^m)\right|$ for any $m\le 3$. 
So $H=\SL_5(4^2)$.

When $q=2^2$, we take $b=c=\omega$. In this case, the classes  $\mathcal{C}_6$ and $\mathcal{S}$ are empty. By Lemmas \ref{irred5}  and    \ref{monomial5}  the subgroup $H$ is irreducible and  not monomial. Furthermore,  $\F_4=\F_2(\omega)$ and by Lemma \ref{forms2} $H$ 
does not fix non-degenerate forms. Again $H=\SL_5(2^2)$. 
\end{proof}

We recall that the group $\PSU_5(4)$ is not $(2,3)$-generated, see \cite{V}.

\begin{theorem}\label{positive results5}  Let $x,y$ be as in \eqref{generators5} with $c\in \F_{q^2}^ \ast$,
$b=c^q - c -1 $ and suppose:
\begin{itemize}
\item[\rm{(i)}] $c^{2q}+c^{q+1}-3c^q+c^2-3c+3\ne 0$;
\item[\rm{(ii)}] $c^{2q-1}-2c^{q}+8c^{q-1}+c+8\ne 0$; 
\item[\rm{(iii)}] $\F_{q^2}=\F_p[c^5]$.
\end{itemize}

Then $H=\SU_5(q^2)$. Moreover, if  $q^2\ne 2^2, 4^ 2$, there exists $c\in \F_{q^2}^\ast$ satisfying
conditions {\rm (i)} to {\rm (iii)}.  
If $q^2=4^ 2$, there exists $c\in \F_{q^2}^\ast$ such that $H=\SU_5(16)$.
In particular, the groups $\SU_5(q^2)$ and $\PSU_5(q^2)$
are $(2,3)$-generated for all  $q^2\ge 9$. 
\end{theorem}

\begin{proof}
$H\le \SU_5(q^2)$ by  Lemma  \ref{forms2} and the assumption $b=c^q - c -1$.
Let $M$ be a maximal subgroup of $\SU_5(q^2)$ which contains $H$. 
Conditions (i) and (ii) correspond to those in Lemma \ref{irred5}, setting $b=c^q-c-1$.
Thus  $H$ is absolutely irreducible.  In particular $M\not\in\mathcal{C}_1\cup\mathcal{C}_3$.
Moreover $M$ cannot belong to $\mathcal{C}_2$  by (iii) and Lemma \ref{monomial5}.
Since $z^6$ is not scalar by Lemma \ref{5powers}, we may apply Lemma
 \ref{C6} with $s=5$ to deduce that $M\notin\mathcal{C}_6$.
We have $\F_p[c^{5q}]=\F_p[c^5]=\F_{q^2}$, by (iii). 
Since $b+c+1=c^q$ is  a coefficient of the characteristic polynomial of $z$
(see \eqref{char5}), the matrix $z$ cannot be conjugate 
to any element of $\SL_5(q_0)Z$, with $q_0< q^2$ and $Z$ the center of $\SU_5(q^ 2)$.
Thus $M\not\in \mathcal{C}_5$. 
Finally,  $M$ cannot be in class $\mathcal{S}$ by Lemma \ref{class S}.
We conclude that $H=\SU_5(q^2)$.

As to the existence of some $c$ satisfying all the assumptions, suppose $q^2\ne 16$.
Any element of $\F_{q^2}^ \ast$ of order $q^2-1$  satisfies (iii) by Lemma \ref{C3}.
The non-zero elements in $\F_{q^2}$ which do not satisfy either (i) or (ii)  are
at most  $2q+2q-1=4q-1$.  If $q \geq 16$,  
by Lemma \ref{pellegrini} there are at least 
$4q$ elements in $\F_{q^2}^ \ast$ having order $q^2-1$. So $c$ exists.
For $q=4,7,8,9,11,13$, we may take $c$ with minimal polynomial $m_c$ as in Table B and
for $q=3,5$ we take $c$ of order $q^2-1$, with minimal polynomial $m_c=t^2-t-1$ and $m_c=t^2-t+2$, respectively.
Then $c$ satisfies (i) and (ii) since $m_c$ is coprime with 
$$p_1=t^{2q}+t^{q+1}-3t^q+t^2-3t+3,\quad p_2=t^{2q-1}-2t^{q}+8t^{q-1}+t+8.$$
It also satisfies (iii) except for $q=4$. In this case, $z$ has order divisible by $17$, a prime which does not divide $\left|\SU_5(2^2)\right|$. Thus $H=\SU_5(q^2)$, for all $q\neq 2$.
\end{proof}

\end{document}